\documentclass[12pt]{amsart}

\usepackage{latexsym}
\usepackage{amssymb, amsfonts, amsthm, amsmath}
\usepackage{hyperref}
\usepackage{amscd}

\theoremstyle{theorem}
 \newtheorem{thm}{Theorem}[section]
 \newtheorem{prop}[thm]{Proposition}
 \newtheorem{cor}[thm]{Corollary}

\theoremstyle{remark}
  \newtheorem{rem}{Remark}[section]
  
 \newtheorem{eg}{Example}[section]
 
\theoremstyle{definition}
  \newtheorem{defn}{Definition}[section]

\newtheoremstyle{Acknowledgements}
  {}
    {}
     {}
     {}
    {\bfseries}
    {}
     {.5em}
     {\thmname{#1}\thmnumber{ }\thmnote{ (#3)}}
\theoremstyle{Acknowledgements}
\newtheorem{ack}{Acknowledgements}
  
\newcommand{\done}{\hfill $\triangleleft$}

\newcommand{\FF}{\mathbb{F}}
\newcommand{\CC}{\mathbb{C}}

\newcommand{\kk}{\Bbbk}

\newcommand{\ii}{\mathcal{I}}
\newcommand{\s}{\mathcal{S}_G}

\newcommand{\pp}{\mathfrak{p}}

\newcommand{\kvg}{{\kk[V]^G}}
\newcommand{\kvgplus}{{\kk[V]^G_+}}

\newcommand{\ahat}{\widehat{A}}
\newcommand{\kvghat}{\widehat{\kvg}}
\newcommand{\kvhat}{\widehat{\kv}}
\newcommand{\vhatg}{\widehat{V}/\!\!/G}
\newcommand{\vhat}{\widehat{V}}
\newcommand{\pihat}{\widehat{\pi}}

\newcommand{\kv}{{\kk[V]}}

\newcommand{\vg}{{V/\! \!/G}}

\newcommand{\kvgbar}{\overline{\kk}[\overline{V}]^G}
\newcommand{\vbar}{\overline{V}}
\newcommand{\kbar}{\overline{\kk}}
\newcommand{\kvbar}{\overline{\kk}[\overline{V}]}

\newcommand{\spec}{\mathrm{Spec}}
\newcommand{\red}{\mathrm{red}}

\newcommand{\codim}{\mathrm{codim}}

\author{Emilie~Dufresne}
\address{Department of Mathematics and Statistics\\
                 Queen's University\\
                 Jeffery Hall, University Avenue\\
                 Kingston, ON, Canada\\
                 K7L 3N6}
\email{emilie@mast.queensu.ca}
\title[Separating Invariants and Finite Reflection Groups]{ Separating Invariants \\and Finite Reflection Groups}

\begin{document}

\bibliographystyle{plain}


\begin{abstract}

 A separating algebra is, roughly speaking, a subalgebra of the ring of invariants whose elements distinguish between any two orbits that can be distinguished using invariants.  In this paper, we introduce  a geometric notion of separating algebra. This allows us to prove that only groups generated by reflections may have  polynomial separating algebras, and only groups generated by bireflections may have complete intersection separating algebras.
\end{abstract}


\maketitle


\section{Introduction}

The study of separating invariants is a new trend in Invariant Theory initiated by Derksen and Kemper \cite{hd-gk:cit,gk:cirgpc}. It returns to the roots of Invariant Theory: using invariants to distinguish between the orbits of a group action on some geometric or algebraic space. Roughly speaking, a separating algebra is a subalgebra of the ring of invariants whose elements distinguish between any two orbits that can be distinguished using invariants. A separating set need not generate the ring of invariants. Separating algebras are often much better behaved than the ring of invariants. In contrast with the ring of invariants, a finitely generated separating algebra always exists, with no restrictions on the group (Theorem~2.3.15 in \cite{hd-gk:cit}); the polarization of separating invariants yields a separating set (Theorem 1.4 in \cite{jd-gk-dw:psi}); for finite groups, the Noether bound holds for separating invariants (Corollary~3.9.14 in \cite{hd-gk:cit}). 

In this paper, we introduce the notion of geometric separating algebra, a notion of separating algebra stable under extensions of the base field. But more importantly, we give two geometric formulations of this notion. These provide the keys to relating the structure of geometric separating algebras to the geometry of the representation.

Let $G$ be a group, and let $V$ be a finite dimensional representation of $G$ over a field $\kk$. Our first two results link the existence of polynomial geometric separating algebras to the presence of \emph{reflections}, elements of $G$ fixing a subspace of codimension 1  in $V$:

\begin{thm}\label{polysepalg->refngr}
Let $G$ be a finite group. If there exists a geometric separating algebra which is a polynomial ring, then the action of $G$  on $V$ is generated by  reflections.\end{thm}

Our method provides a new proof for the result of Serre~\cite{jps:gfaalr} which established that only reflection groups may have polynomial rings of invariants. In Example~\ref{MyEg}  the ring of invariants is not polynomial, but  we give a polynomial geometric separating algebra. Thus, Theorem~\ref{polysepalg->refngr} is a strict generalization of Serre's result.

An interesting consequence is a characterization, when the order of the group is invertible in the base field, of the existence of polynomial geometric separating algebras. It generalizes the well-known result of  Shephard and Todd~\cite{gcs-jat:furg}, Chevalley~\cite{cc:ifggr}, Serre, and Clark and Ewing~\cite{ac-je:trpacr} which states that when the order of the group is invertible in the base field, the ring of invariants is polynomial if and only if the group is generated by reflections.  
 \begin{cor}\label{polysepalg<->refngr}
Let $G$ be a finite group. If the characteristic of $\kk$ does not divide the order of $G$, then there exists a geometric separating algebra which is a polynomial ring if and only if the action of $G$ on $V$ is generated by reflections.
\end{cor}

Corollary \ref{polysepalg<->refngr} is new, even in characteristic zero. A version of it is an easy consequence of Theorem 1.6 of  \cite{gk:cirgpc} (or Theorem 2.3.12 of \cite{hd-gk:cit}), but it requires the additional assumption that the polynomial separating algebra be graded and the base field algebraically closed.

Our third result generalizes a result of Gordeev \cite{vk-kw:flgrici}, and Kac and Watanabe  \cite{vk-kw:flgrici}. They proved that when the ring of invariants is a complete intersection, the group must be generated by \emph{bireflections}, that is, elements of $G$ fixing a subspace of codimension 2 in $V$. 
 
\begin{thm}\label{CISepAlg->bir}
Let $G$ be a finite group. If there exists a graded geometric separating algebra which is a complete intersection, then the action of $G$ on $V$ is generated by bireflections.
\end{thm}

Examples where there is a geometric separating algebra which is a complete intersection abound. For example, for every  2-dimensional representations of finite abelian groups over $\CC$, there is  a geometric separating algebra which is a hypersurface (see \cite{ed:si}). The abundance of similarly well-behaved examples led to the question of whether there always existed such nice separating algebras (see the introduction of \cite{gk:si}).  Example \ref{DerksenEg}, due to Harm Derksen, shows that, in general, we can not expect an hypersurface geometric separating algebra to exist. His example, however, is of a representation of $\CC^*$, and the technique he uses does not appear to be adaptable to finite groups. Thus, the question remained open for finite groups. With Theorem~\ref{CISepAlg->bir}, we provide some answers: if $G$ is not generated by bireflections, no graded geometric separating algebra is a complete intersection, and in particular,  no graded geometric separating algebra is a hypersurface.

Section 2 of this paper introduces the notion of geometric separating algebra, and more importantly, its two geometric formulations. In Section 3, we prove Theorems  \ref{polysepalg->refngr} and \ref{polysepalg<->refngr}, and then discuss briefly some interesting examples, and questions arising from them. In Section 4, we prove Theorem \ref{CISepAlg->bir}, and end with Harm Derksen's example.

\begin{ack}
I thank Harm Derksen for giving permission to include his example (Example \ref{DerksenEg}). I also thank David Wehlau for his supervision during my graduate studies, and Mike Roth and Greg Smith for helpful conversations.
\end{ack}


\section{Geometric Separating Invariants}\label{defn}

This section establishes the definitions and results needed throughout this text. It introduces the notion of geometric separating algebra and its two geometric formulations.
Consider a linear algebraic group $G$, and  a $n$-dimensional representation $V$ of $G$  over a field $\kk$. We write $\kv$ for the symmetric algebra on the vector space dual of $V$. The polynomial ring $\kv$ has the standard grading.  The action of the group $G$ on $V$ induces an action on $\kv$. Specifically, for a polynomial $f$,  the action of an element $\sigma$ of $G$,  is given by $(\sigma\cdot f)(u)=f(\sigma^{-1}\cdot u)$, for $u$ in $V$. We let $\kvg$ denote the ring formed by the elements of $\kv$ fixed by the $G$-action. Since the group action preserves degree, $\kvg$ is a graded subalgebra of $\kv$.

By definition, elements of $\kvg$ are constant on $G$-orbits. Thus, if an invariant $f$ takes distinct values on elements $u$ and $v$ of $V$, then we know that these elements belong to distinct orbits, that is, $f$ \emph{separates} $u$ and $v$. A naive definition for a separating set would be to require that it separates elements $u$ and $v$ whenever they belong to distinct orbits. The ring of invariants, however, often does not distinguish the orbits (see, for example, Section 2.3.1 in \cite{hd-gk:cit}). Derksen and Kemper~ \cite{hd-gk:cit} and \cite{gk:cirgpc}  define a \emph{separating set} as a subset $E$ of $\kvg$ such that for all $u,v$ in $V$, if there exists $f$ in $\kvg$ with $f(u)\neq f(v)$, then there exists $h$ in $E$ with $h(u)\neq h(v)$.

This definition yields interesting results: \cite{gk:cirgpc} uses the computation of a separating set as an intermediate step in the computation of a generating set for the invariants of reductive groups in positive characteristic; \cite{md:tsi,jd-gk-dw:psi} show the polarization of separating sets yields separating sets; \cite{gk:si} offers a generalization of separating sets to more general rings of functions. Many of these results, however, require the base field to be algebraically closed;  this notion of separating algebra behaves rather differently over non algebraically closed fields, and its behaviour over finite fields diverges even more from the situation over algebraically closed fields. The definition we suggest in the present paper is stable under extensions of the base field.

Let $\kbar$ be an algebraic closure of the field $\kk$, and let $\vbar=V\otimes_\kk{\kbar}$, then $\kv\subset\kvbar$, and so any  $f$ in $\kv$ can be considered as a function $\vbar\rightarrow{\kbar}$. Moreover, we can extend the action of $G$ on $V$ to an action of $G$ on $\vbar$. Thus, we can view $\kvg\subset\kvgbar$.

\begin{defn}\label{GeomSepAlgDefn}
A subset $E$ of $\kvg$ is a \emph{geometric separating set} if, for all $u$ and $v$ in $\vbar$,  the two following equivalent statements hold:
\begin{itemize}
	\item if there exists $f$ in $\kvg$ such that $f(u)\neq f(v)$, then there exists $h$ in $E$ such that $h(u)\neq h(v)$;
	\item $f(u)=f(v)$, for all $f$ in $\kvg$ if and only if $h(u)=h(v)$ for all $h$ in $E$.
\end{itemize}
A subalgebra $A\subset\kvg$ satisfying these conditions is called a \emph{geometric separating algebra}. If a subalgebra of $\kvg$ is generated by a geometric separating set, then it is a geometric separating algebra.\end{defn}

The adjective ``Geometric'' is often used to describe similar constructions. The strength of this new definition lies in the two geometric formulations presented below.  In fact, for most of the results concerning separating invariants found in the literature, a geometric separating invariants version holds, often removing the requirement on $\kk$ to be algebraically closed (see \cite{ed:si}).



Our first geometric formulation of the definition of geometric separating algebra is a generalization to general groups and fields of some ideas of Kemper   for reductive groups over algebraically closed fields (Section 2 of \cite{gk:cirgpc}). 
We write $V$ for the affine scheme corresponding to $\kv$, $\vg$ for the affine scheme corresponding to $\kvg$, and $\pi$ for the morphism  from $V$ to $\vg$ corresponding to the inclusion $\kvg\subset\kv$. 

\begin{defn} 
The separating scheme $\s$ is the unique reduced scheme having the same underlying topological space as the product $V\times_\vg V$, that is, $\s:=(V\times_\vg V)_\red$.
\end{defn}

The separating scheme can be used to detect when a subalgebra $A\subset\kvg$ is a geometric separating algebra:

\begin{thm}\label{SepDefn}
Let $A\subset\kvg$ be a subalgebra, then the following statements are equivalent:

\begin{enumerate}

\item \label{defnA} A is a geometric separating algebra;

\item \label{geom} if $W=\spec(A)$, then the natural morphism
$\s\rightarrow(V\times_W V)_\red$
is an isomorphism;

\item\label{alge} 
if $\delta$ denotes the map $\delta: \kv \longrightarrow \kv\otimes_\kk \kv$ sending an element $f$ of $\kv$ to $f\otimes1-1\otimes f$,
then the ideals $(\delta(A))$ and $(\delta(\kvg))$ have the same radical in the ring $\kv\otimes_\kk\kv$, i.e.,
\[\sqrt{\delta(A)}=\sqrt{\delta(\kvg)};\]\end{enumerate}

\begin{proof}
First, we prove (\ref{geom}) and (\ref{alge}) are equivalent. As $V$, $\vg$, and $W$ are affine schemes, we have
$V\times_\vg V=\spec(\kv\otimes_\kvg\kv) $,
and
$V\times_W V=\spec(\kv\otimes_A\kv)$.
For any subalgebra $B\subset\kv$,
\[\kv\otimes_B\kv=\frac{\kv\otimes_\kk\kv}{(\delta(B))}.\]
Thus, (\ref{geom}) is equivalent to saying the $\kk$-algebra homomorphism
\[\frac{\kv\otimes_\kk\kv}{\sqrt{\delta(A)}}\rightarrow \frac{\kv\otimes_\kk\kv}{\sqrt{\delta(\kvg)}}\]
is an isomorphism, that is, $\sqrt{\delta(A)}=\sqrt{\delta(\kvg)}$.

We now prove (\ref{defnA}) and (\ref{alge}) are equivalent. If $A$ is a geometric separating algebra, then for any $u$ and $v$ in $\overline{V}$, $f(u)=f(v)$ for all $f$ in $\kvg$ if and only if $h(u)=h(v)$ for all $h$ in $A$. If  $\ii_{\vbar^2}(u,v)$ denotes the maximal ideal of $\kvbar\otimes_{\kbar}\kvbar$ corresponding to the point $(u,v)$ of $\vbar\times\vbar$, then we can rewrite this statement as:
\[\ii_{\vbar^2}(u,v)\cap\left(\kv\otimes_\kk\kv\right)\supset\delta(\kvg)\]
if and only if
\[\ii_{\vbar^2}(u,v)\cap\left(\kv\otimes_\kk\kv\right)\supset\delta(A).\]

Since the maximal ideals of $\kv\otimes_\kk\kv$ are in bijection with Galois orbits of the maximal ideals of $\kbar[\vbar]\otimes_{\kbar} \kbar[\vbar]$, the maximal ideals of $\kv\otimes_\kk\kv$ are exactly the primes of the form
$$\ii_{\vbar^2}(u,v)\cap\left(\kv\otimes_\kk\kv\right).$$
As $\kv\otimes_\kk\kv$ is a finitely generated $\kk$-algebra, the radical of an ideal $I$ is given by the intersection of all maximal ideals containing $I$ (Theorem 5.5 of \cite{hm:crt}). 
Therefore, $\sqrt{\delta(\kvg)}=\sqrt{\delta(A)}$.
\end{proof}\end{thm}

\begin{rem}
The proof of Theorem \ref{SepDefn} implies that a subset $E\subset\kvg$ is a geometric separating set if and only if  $\sqrt{\delta(\kvg)}=\sqrt{\delta(E)}$.
\end{rem}


Under an additional hypothesis, we obtain another geometric formulation of geometric separation which involves the notion of radicial morphism (Definition 3.5.4 in \cite{ag:egai}). A map of schemes $f:X\rightarrow Y$ is \emph{radicial} if for any field $\FF$,  the corresponding map of $\FF$-points is injective.

\begin{thm}\label{SepDefn2}
If $G$ is reductive, then a subalgebra $A\subset\kvg$ is a geometric separating algebra if and only if the morphism of schemes $\theta:\vg\rightarrow W=\spec(A)$ corresponding to the inclusion $A\subset\kvg$ is a radicial morphism.
\begin{proof}

 We write $\gamma$ for the morphism of schemes corresponding to the inclusion  $\kv\subset\kvbar$. By definition, a subalgebra $A\subset\kvg$ is a geometric separating algebra if and only if for $u$ and $v$ in $\vbar$, having $\theta(\pi(\gamma(u)))=\theta(\pi(\gamma(v)))$ implies that $\pi(\gamma(u))=\pi(\gamma(v))$. In other words, $A$ is a geometric separating algebra if and only if $\theta$ is injective on $\kbar$-points in the image of $\vbar$. 
 
 On the other hand, since $G$ is reductive, $\pi$ is surjective (Lemma~1.3 in \cite{gk:aclrgi}), and any map $\spec(\kbar)\rightarrow\vg$ factors through $V$. Since $\vbar\rightarrow V$ is also surjective, $\spec(\kbar)\rightarrow V$ factors through $\vbar$. Thus, all $\kbar$-points of $\vg$ are in the image of $\vbar$. Therefore, $A$ is a geometric separating algebra if and only if $\theta$ is injective on all $\kbar$-points.  But since $V$ and $W$ are of finite type over $\kbar$, by Propositions 1.8.4 and 1.8.7.1 of \cite{ag:egaiva} injectivity on $\kbar$-points is equivalent to injectivity on any $\mathbb{F}$-points. It follows that $A$ is a geometric separating algebra if and only if $\theta$ is radicial.\end{proof}\end{thm} 

\begin{rem} \label{samedim} It follows from Theorem \ref{SepDefn2} that if we assume $G$ is reductive, then finitely generated geometric separating algebras have the same dimension as the ring of invariants.
\end{rem}


\section{ Polynomial Geometric Separating Algebras}\label{poly}

In this section, we prove Theorem \ref{polysepalg->refngr} and Corollary~\ref{polysepalg<->refngr}, and then we provide non-trivial examples where  there is a geometric separating algebra which is a polynomial ring. The following concrete description of the separating scheme for finite groups is a key element in our proof:

\begin{prop}\label{SG}
If $G$ is a finite group,  then the separating scheme is a union of $|G|$ linear subspaces, each of dimension $n$. There is a natural correspondence between these linear spaces and the elements of $G$. Moreover, if $H_\sigma$ and $H_\tau$ denote the subspaces corresponding to the elements $\sigma$ and $\tau$ of $G$, respectively, then the dimension of the intersection $H_\sigma\cap H_\tau$ is equal to the dimension of the subspace fixed by $\tau^{-1}\sigma$ in $V$.
\begin{proof}

 For each $\sigma\in G$, let $H_\sigma$ be the graph of $\sigma$, that is
 \[H_\sigma=\{(u,\sigma\cdot u)\in V\times V \mid u\in V\}.\]
The linear space $H_\sigma$ has dimension $n$. For elements $\sigma$ and $\tau$ of $G$, the intersection $H_\sigma\cap H_\tau$ is 
 $\{(u,v)\in V\times V \mid u\in V\mathrm{~and~} v=\sigma\cdot u=\tau\cdot u \}.$
Hence, $H_\sigma\cap H_\tau$ is isomorphic to the fixed space of $\tau^{-1}\sigma$. Next, we show that $$\s=\bigcup_{\sigma\in G}H_\sigma.$$
For each $\sigma\in G$, $H_\sigma$ is given as a closed subscheme of $V\times V$ by the ideal $(f\otimes1-1\otimes\sigma^{-1}f\mid f\in\kv)$. Thus, in algebraic terms,  we want to show that
\[ \frac{\kv\otimes_\kk\kv}{\sqrt{(\delta(\kvg))}}=\frac{\kv\otimes_\kk\kv}{\bigcap_{\sigma\in G}(f\otimes1-1\otimes\sigma^{-1}f\mid f\in\kv)}.\]
As $G$ is finite, the ring of invariants separates orbits in $\vbar$ (Lemma~2.1 in \cite{jd-gk-dw:psi}), thus for $u$ and $v$ in $\vbar$, $f(u)=f(v)$, for all $f$ in $\kvg$, if and only if there exists $\sigma$ in $G$ such that $u=\sigma v$. In other words, $\delta(\kvg)\subset \ii_{\vbar^2}(u,v)\cap(\kv\otimes_\kk\kv),$
if and only if
$$\bigcap_{\sigma\in G}(f\otimes1-1\otimes\sigma^{-1}f\mid f\in\kv)\subset  \ii_{\vbar^2}(u,v)\cap(\kv\otimes_\kk\kv).$$
Therefore, $\sqrt{(\delta(\kvg))}=\bigcap_{\sigma\in G}(f\otimes1-1\otimes\sigma^{-1}f\mid f\in\kv)$.\end{proof}\end{prop}

We may now prove the two main results of this section.

\begin{proof}[Proof of Theorem \ref{polysepalg->refngr}]
Suppose a separating algebra $A$ is a polynomial ring. By Remark \ref{samedim}, $A$ is $n$-dimensional, thus $A$ is generated by $n$ elements. It follows that  the ideal $(\delta(A))$ is also generated by $n$ elements. Therefore, $V\times_WV$ is a complete intersection, and in particular, it is Cohen-Macaulay. As $V\times_WV$ is Noetherian, Hartshorne's Connectedness Theorem (Corollary 2.4 in \cite{rh:cic}) implies that $V\times_WV$ is connected in codimension 1, and thus, so is $\s=(V\times_WV)_\red$. 

Consider the irreducible components $H_{1}$ and $H_\sigma$ of $\s$ corresponding to the identity and an arbitrary element $\sigma$ of $G$, respectively. As $\s$ is connected in codimension 1, there is a sequence of irreducible components
$$H_{1}=H_{\sigma_0},\cdots,H_{\sigma_r}=H_\sigma,$$
such that $H_{\sigma_i}\cap H_{\sigma_{i+1}}$ has codimension 1. By Proposition \ref{SG},   $\sigma_i^{-1}\sigma_{i+1}$ fixes a subspace of codimension~1, and so it acts as a reflection on $V$. Thus,
$\sigma=1^{-1}\sigma=\sigma_0^{-1}\sigma_r=(\sigma_0^{-1}\sigma_1)(\sigma_1^{-1}\sigma_2)\cdots(\sigma_{r-1}^{-1}\sigma_r)$
is a product of  reflections on $V$. Therefore, the action of $G$ on $V$ is generated by reflections.\end{proof}

\begin{proof}[Proof of Theorem \ref{polysepalg<->refngr}] One direction is given by Theorem \ref{polysepalg->refngr}, and as the ring of invariants is a geometric separating algebra, the other is an immediate consequence of the part of the result of Shephard and Todd \cite{gcs-jat:furg}, Chevalley \cite{cc:ifggr}, Serre, and Clark and Ewing \cite{ac-je:trpacr} which establishes that reflection groups have polynomial ring of invariants.\end{proof}

 The following example shows that it is possible for a geometric separating algebra to be polynomial even if the ring of invariants is not, showing that Theorem \ref{polysepalg->refngr} is stronger than Serre's result.
 
 \begin{eg}\label{MyEg}
Let $\kk$ be a field of characteristic $p$, containing a root $z$ of the irreducible polynomial $Z^p-Z-1$. Let
\renewcommand{\arraystretch}{0.8}
\renewcommand{\arraycolsep}{2pt}
\[G={
			\left\langle
			\left( \begin{array}{cccc}
			1 & 0 & 0 & 0\\
			0 & 1 & 0 & 0\\
			0 & 0 & 1 & 0\\
			0 & 0 & 1 & 1\end{array}\right),
			\left( \begin{array}{cccc}
			1 & 0 & 0 & 0\\
			1 & 1 & 0 & 0\\
			0 & 0 & 1 & 0\\
			0 & 0 & 0 & 1\end{array}\right),
			\left( \begin{array}{cccc}
			1 & 0 & 0 & 0\\
			z & 1 & 0 & 0\\
			0 & 0 & 1 & 0\\
			1 & 0 & 0 & 1\end{array}\right) \right\rangle}.\]

Let $\{x_1, y_1, x_2, y_2\}$ be the dual basis for $V^*$. Using Proposition~3.1 of \cite{heac-iph:ricpgoffp} (see \cite{ed:si} for details) we obtain that the ring of invariants of $G$ is generated minimally by $x_1$, $x_2$, $M_1$, $M_2$, and $h$, where
\[\begin{array}{lcl}
M_1&=&(y_1^p-x_1^{p-1}y_1)^p-(x_1^p)^{p-1}(y_1^p-x_1^{p-1}y_1),\\ 
M_2&=&(y_2^p-x_2^{p-1}y_2)^p-(x_1^{p}-x_2^{p-1}x_1)^{p-1}(y_2^p-x_2^{p-1}y_2),\\ 
h&=&(x_1^{p-1}-x_2^{p-1})(y_1^p-x_1^{p-1}y_1)-x_1^{p-1}(y_2^p-x_2^{p-1}y_2).\\ 
\end{array}\]
Thus, the ring of invariants $\kvg$ is a hypersurface, and  a generating relation is given by
\[h^p-(x_1^{p-1}-x_2^{p-1})^pM_1+x_1^{p^2-p}M_2-(x_1^p(x_1^{p-1}-x_2^{p-1}))^{p-1}h=0.\]
We can rewrite this relation as
\[h^p= (x_1^{p-1}-x_2^{p-1})^pM_1-x_1^{p^2-p}(M_2-(x_1^{p-1}-x_2^{p-1})^{p-1}h).\]
As $h$ and $h^p$ separate the same points, $$S=\{x_1,x_2,M_1,M_2-(x_1^{p-1}-x_2^{p-1})^{p-1}h\}$$ is a geometric separating set which generates a polynomial geometric separating algebra.
\done
\end{eg}

In \cite{ed:si} more examples are  discussed, including another similar infinite family. As required by Theorem \ref{polysepalg->refngr},  all the group actions involved  are generated by reflections. But they have more in common: they act as rigid groups, i.e., the actions of all  their isotropy subgroups are generated by reflections. We do not expect polynomial geometric separating algebras to exist for all rigid groups. On the other hand, we suspect that rigid groups are the only ones for which polynomial separating algebras can exist.


                      

\section{Complete Intersection Geometric Separating Algebras}\label{ci}

In this section, we prove Theorem \ref{CISepAlg->bir}, and then present an example. Our proof extends the argument used by Kac and Watanabe to prove their Theorem A in \cite{vk-kw:flgrici}, and it exploits the second geometric formulation of the notion of geometric separating algebra. The statement of Theorem \ref{CISepAlg->bir} concerns graded separating algebras. Assuming that a geometric separating algebra $A$ is graded imposes a very close relationship with the ring of invariants $\kvg$:

\begin{prop}\label{inj->finite}
Let $A\subset\kvg$ be a graded subalgebra. If the map of schemes $\theta:\vg\rightarrow W=\spec(A)$ is injective, then the extension $A\subset\kvg$ is integral.
\begin{proof}
Let $A_+$ and $\kvgplus$ denote the maximal graded ideals of $A$ and $\kvg$, respectively. If $\pp$ is a proper prime ideal of $\kvg$ containing $A_+\kvg$, then 
\[ A_+\subset A_+\kvg\cap A\subset\pp\cap A\subset A.\]
As $A_+$ is a maximal ideal, $\pp\cap A=A_+$. 

On the other hand, $A_+=\kvgplus\cap A.$ Thus, the injectivity of $\theta$ implies $\pp=\kvgplus$, and the radical of ${A_+\kvg}$ in $\kvg$ is $\kvgplus$. It follows that ${\kvg}/{A_+\kvg}$ has Krull dimension $0$, i.e., it is a finite dimensional $\kk$-vector space. By the graded version of Nakayama's Lemma (Lemma~3.5.1 in \cite{hd-gk:cit}), $\kvg$ is a finite $A$-module, and so the extension $A\subset\kvg$ is integral. 
\end{proof}\end{prop}

\begin{cor}\label{sep-int}
If the action of $G$ on $V$ is reductive, and if $A\subset\kvg$ is a graded geometric separating algebra, then the extension $A\subset\kvg$ is integral.
\begin{proof} By Theorem \ref{SepDefn2}, the morphism of schemes $\theta:\vg\rightarrow W$ is radicial.  As radicial morphism of schemes are injective (Proposition~3.5.8 in \cite{ag:egai}),  and as $A$ is assumed to be graded,  the corollary follows directly from Proposition \ref{inj->finite}.\end{proof}\end{cor}

We can now prove the main result of this section:

\begin{proof}[Proof of Theorem \ref{CISepAlg->bir}]
Without loss of generality, we may assume that the base field is algebraically closed. Indeed, if $A$ is a complete intersection graded geometric separating algebra inside of $\kvg$, then $A\otimes_\kk\kbar$ is a complete intersection and a graded geometric separating algebra inside of $\kbar[\vbar]^G$. Assuming the theorem holds over algebraically closed fields, it follows that $G$ is generated by bireflection on $\vbar$. Thus, the action of $G$ on $V$ is also generated by bireflections. 

Since $G$ is finite, it is reductive, and so Theorem \ref{SepDefn2} implies that  $\theta$  is a radicial morphism. As $A$ is graded,  Corollary \ref{sep-int} implies $\theta$ is finite. Finally, since $\theta$ is dominant and finite it is also surjective.

Now, let $\kvhat$, $\kvghat$\!\!, and $\ahat$ be the completions of $\kv$, $\kvg$\!\!, and $A$ at their maximal graded ideal $\kv_+$, $\kvgplus$, and $A_+$, respectively. A scheme is \emph{simply connected} if and only if there are no nontrivial \'etale coverings (\cite{rh:ag}, Example 2.5.3). As complete local rings satisfy Hensel's Lemma, by Theorem 5.2 in \cite{bi:glsmca}, the affine schemes corresponding to the completions are simply connected.  The $G$-action on $\kv$ extends to a $G$-action on $\kvhat$, and $\kvghat=({\kvhat})^G$. Thus $\spec(\kvghat)=\vhatg$, and the finite morphism $\pi$ lifts to the quotient morphism $\pihat:\vhat\rightarrow\vhatg$, which remains finite. Since $\kvghat=\kvg\otimes_A\ahat$ (Theorem 9.3A in \cite{rh:ag}), taking the completion corresponds to doing a base change. Hence, $\theta$ lifts to a morphism $\widehat{\theta}$ which is surjective, radicial, and finite, since all three properties are preserved by base changes: see Propositions 3.5.2 and 3.5.7 in \cite{ag:egai}, and 6.1.5 in \cite{ag:egaii}, respectively.

For $\sigma$ in $G$ we let $\vhat^\sigma$ denote the subscheme of fixed points of $\sigma$ on $\vhat$. Let $L$ be the union of all the $\vhat^\sigma$'s with codimension at least 3, and put $M=\pihat(L)$, and $N=\widehat{\theta}(M)$.  Since $W$ is a complete intersection, Proposition 3.2 of \cite{rh:cic} implies  $\widehat{W}=\spec(\ahat)$ is also a complete intersection. Since $\pihat$ and $\widehat{\theta}$ are finite, $N$ has codimension 3 in $\widehat{W}$. Hence, by Lemma 1 from \cite{vk-kw:flgrici},  $\widehat{W}\setminus N$ is simply connected. As the restriction of $\widehat{\theta}$ to $\vhatg\setminus M$ is radicial, surjective, and finite, by Theorem 4.10 of \cite{ag:sga1b}, it follows that $(\vhatg)\setminus M$ is also simply connected.

Furthermore, $X=\vhat\setminus L$ is an integral scheme with the induced $G$-action, and $(\vhatg)\setminus N=X/\!\!/G$. Since $X/\!\!/G$ is simply connected,  Lemma 2 from \cite{vk-kw:flgrici} implies that the group $G$ is generated by the set $\{G_x \mid x\in X=\vhat\setminus L\}$. But by the definition of $\vhat\setminus L$, an element $\sigma$ belongs to $G_x$ for some $x\in \vhat\setminus L$  if and only if $\codim( \vhat^\sigma) \leq 2$. Hence, $G$ is generated by bireflections. 
\end{proof}

Derksen's example of a representation for which no geometric separating algebra is a hypersurface follows:

\begin{eg}[Harm Derksen]\label{DerksenEg}
Let $t$ in $G=\CC^*$ act on the polynomial ring $\CC[x_1,x_2,x_3,y_1,y_2]$, as
\renewcommand{\arraystretch}{0.9}
\renewcommand{\arraycolsep}{3pt}
\[{\left(\begin{array}{ccccc}
t & 0 & 0 & 0      & 0\\
0 & t & 0 & 0      & 0\\
0 & 0 & t & 0      & 0\\
0 & 0 & 0 & t^{-1} & 0\\
0 & 0 & 0 & 0      & t^{-1}\end{array}\right).}\]
Monomials are sent to scalar multiples of themselves, and so the ring of invariants is generated by monomials. In fact,
\[\CC[V]^{\CC^*}=\CC[x_1y_1,x_2y_1,x_3y_1,x_1y_2,x_2y_2,x_3y_2].\]
The dimension of $\CC[V]^{\CC^*}$ is equal to its transcendence degree (i.e., the maximal number of algebraically independent elements). The set $\{x_1y_1,x_3y_1,x_1y_2,x_2y_2\}$ forms a transcendence basis for $\CC[V]^{\CC^*}$. Indeed, they are clearly algebraically independent, and  the relations
$(x_1y_1)(x_3y_2)=(x_3y_1)(x_1y_2)$ and $(x_2y_2)(x_1y_1)=(x_2y_1)(x_1y_2)$
 give  $x_3y_2$ and $x_2y_2$ as roots of polynomials in the other generators. Thus, $\CC[V]^{\CC^*}$ has dimension 4.

As the group is reductive, by Remark \ref{samedim}, geometric separating algebras have dimension 4. Thus, a geometric separating algebra is a hypersurface if it is generated by 5 elements. We will prove that there are no geometric separating sets of 5 elements.

Suppose, by way of contradiction, that $f_1,f_2,f_3,f_4,f_5$ is a geometric separating set. As the $f_i$'s are invariant, we can write:
\[f_i=F_i(x_1y_1,x_2y_1,x_3y_1,x_1y_2,x_2y_2,x_3y_2),\]
where each $F_i$ is a polynomial in $\CC[z_1,z_2,z_3,z_4,z_5,z_6]$. The ideal generated by the 5 polynomials 
$\label{F} F_i(z_1,z_2,z_3,0,0,0)-F_i(0,0,0,z_4,z_5,z_6)$, for $i=1,\ldots,5$,
corresponds to a subvariety of $\CC^6$ which is either empty or has dimension at least 1. As the point $(0,0,0,0,0,0)$ is a common zero, there are infinitely many solutions, in particular there is a non-zero solution $(a,b,c,d,e,f)$. If
$u=(a,b,c,1,0)$, and $v=(d,e,f,0,1)$,
then for all $i=1,\ldots,5$
\[f_i(u)=F_i(a,b,c,0,0,0)=F_i(0,0,0,d,e,f)=f_i(v),\]
that is, the $f_i$'s do not separate $u$ and $v$. We have, however, 
\[\begin{array}{cc}
x_1y_1(u)=a, & x_1y_1(v)=0,\\
x_2y_1(u)=b, & x_2y_1(v)=0,\\
x_3y_1(u)=c, & x_3y_1(v)=0,\\
x_1y_2(u)=0, & x_1y_2(v)=d,\\
x_2y_2(u)=0, & x_2y_2(v)=e,\\
x_3y_2(u)=0, & x_3y_2(v)=f,\end{array}\]
and as $(a,b,c,d,e,f)$ is nonzero, this is a contradiction. We conclude that no geometric separating algebra is a hypersurface.\done\end{eg}


\bibliography{reference}

\end{document}